\documentclass[12pt]{amsart}
\usepackage{amsmath,amsthm,amsfonts,amssymb,mathrsfs}
\date{\today}

\input xymatrix
\xyoption{all}

\usepackage{color}

\newtheorem{theorem}{Theorem}[section]
\newtheorem{proposition}[theorem]{Proposition}
\newtheorem{corollary}[theorem]{Corollary}
\newtheorem{lemma}[theorem]{Lemma}

\theoremstyle{definition}
\newtheorem{remark}[theorem]{Remark}

\usepackage{hyperref}

\begin{document}

\title[On the monoid of monotone injective partial selfmaps of $\mathbb{N}^{2}_{\leqslant}$ with cofinite ...]{On the monoid of monotone injective partial selfmaps of $\mathbb{N}^{2}_{\leqslant}$ with cofinite domains and images}

\author[O.~Gutik and I.~Pozdniakova]{Oleg~Gutik and Inna Pozdniakova}
\address{Department of Mathematics, National University of Lviv, Universytetska 1, Lviv, 79000, Ukraine}
\email{o\_gutik@franko.lviv.ua, ovgutik@yahoo.com, pozdnyakova.inna@gmail.com}



\begin{abstract}
Let $\mathbb{N}^{2}_{\leqslant}$ be the set $\mathbb{N}^{2}$ with the partial order defined as the product of usual order $\leq$ on the set of positive integers $\mathbb{N}$. We study the semigroup $\mathscr{P\!O}\!_{\infty}(\mathbb{N}^2_{\leqslant})$ of monotone injective partial selfmaps of $\mathbb{N}^{2}_{\leqslant}$ having cofinite domain and image. We describe properties of elements of the semigroup $\mathscr{P\!O}\!_{\infty}(\mathbb{N}^2_{\leqslant})$ as monotone partial bijections of $\mathbb{N}^{2}_{\leqslant}$ and show that the group of units of $\mathscr{P\!O}\!_{\infty}(\mathbb{N}^2_{\leqslant})$ is isomorphic to the cyclic group of order two. Also we describe the subsemigroup of idempotents of $\mathscr{P\!O}\!_{\infty}(\mathbb{N}^2_{\leqslant})$ and the Green relations on $\mathscr{P\!O}\!_{\infty}(\mathbb{N}^2_{\leqslant})$. In particular, we show that $\mathscr{D}=\mathscr{J}$ in $\mathscr{P\!O}\!_{\infty}(\mathbb{N}^2_{\leqslant})$.
\end{abstract}

\maketitle


\section{Introduction and preliminaries}

We shall follow the terminology of~\cite{Clifford-Preston-1961-1967} and \cite{Howie-1995}.

In this paper we shall denote the cardinality of the set $A$ by $|A|$.  We shall identify all sets $X$ with their cardinality $|X|$. By $\mathbb{Z}_2$ we shall denote the cyclic group of order two. Also, for infinite subsets $A$ and $B$ of an infinite set $X$ we shall write $A{\subseteq^*}B$ if and only if there exists a finite subset $A_0$ of $A$ such that $A\setminus A_0\subseteq B$.

An algebraic semigroup $S$ is called {\it inverse} if for any element $x\in S$ there exists a unique $x^{-1}\in S$ such that $xx^{-1}x=x$ and $x^{-1}xx^{-1}=x^{-1}$. The element $x^{-1}$ is called the {\it inverse of} $x\in S$.

If $S$ is a semigroup, then we shall denote the subset of idempotents in $S$ by $E(S)$. If $S$ is an inverse semigroup, then $E(S)$ is closed under multiplication and we shall refer to $E(S)$ as a \emph{band} (or the \emph{band of} $S$). If the band $E(S)$ is a non-empty subset of $S$, then the semigroup operation on $S$ determines the following partial order $\leqslant$ on $E(S)$: $e\leqslant f$ if and only if $ef=fe=e$. This order is called the {\em natural partial order} on $E(S)$. A \emph{semilattice} is a commutative semigroup of idempotents. A semilattice $E$ is called {\em linearly ordered} or a \emph{chain} if its natural order is a linear order.

If $S$ is a semigroup, then we shall denote the Green relations on $S$ by $\mathscr{R}$, $\mathscr{L}$, $\mathscr{J}$, $\mathscr{D}$ and $\mathscr{H}$ (see \cite[Section~2.1]{Clifford-Preston-1961-1967}):
\begin{align*}
    &\qquad a\mathscr{R}b \mbox{ if and only if } aS^1=bS^1;\\
    &\qquad a\mathscr{L}b \mbox{ if and only if } S^1a=S^1b;\\
    &\qquad a\mathscr{J}b \mbox{ if and only if } S^1aS^1=S^1bS^1;\\
    &\qquad \mathscr{D}=\mathscr{L}\circ\mathscr{R}=
          \mathscr{R}\circ\mathscr{L};\\
    &\qquad \mathscr{H}=\mathscr{L}\cap\mathscr{R}.
\end{align*}
The $\mathscr{R}$-class (resp., $\mathscr{L}$-, $\mathscr{H}$-, $\mathscr{D}$- or $\mathscr{J}$-class) of the semigroup $S$ which contains an element $a$ of $S$ will be denoted by $R_a$ (resp., $L_a$, $H_a$, $D_a$ or $J_a$).

If $\alpha\colon X\rightharpoonup Y$ is a partial map, then by $\operatorname{dom}\alpha$ and $\operatorname{ran}\alpha$ we denote the domain and the range of $\alpha$, respectively.

Let $\mathscr{I}_\lambda$ denote the set of all partial one-to-one transformations of an infinite set $X$ of cardinality $\lambda$ together with the following semigroup operation: $x(\alpha\beta)=(x\alpha)\beta$ if $x\in\operatorname{dom}(\alpha\beta)=\{ y\in\operatorname{dom}\alpha\mid y\alpha\in\operatorname{dom}\beta\}$,  for $\alpha,\beta\in\mathscr{I}_\lambda$. The semigroup $\mathscr{I}_\lambda$ is called the \emph{symmetric inverse semigroup} over the set $X$~(see \cite[Section~1.9]{Clifford-Preston-1961-1967}). The symmetric inverse semigroup was introduced by Vagner~\cite{Vagner-1952} and it plays a major role in the semigroup theory. An element $\alpha\in\mathscr{I}_\lambda$ is called \emph{cofinite}, if the sets $\lambda\setminus\operatorname{dom}\alpha$ and $\lambda\setminus\operatorname{ran}\alpha$ are finite.

Let $(X,\leqslant)$ be a partially ordered set (a poset). A non-empty subset $A$ of $(X,\leqslant)$ is called a \emph{chain} if the induced partial order from $(X,\leqslant)$ onto $A$ is linear.  For an arbitrary $x\in X$ and non-empty $A\subseteq X$ we denote
 \begin{equation*}
{\uparrow}x=\left\{y\in X\colon x\leqslant y\right\}, \quad {\downarrow}x=\left\{y\in X\colon y\leqslant x\right\}, \quad \uparrow A=\bigcup_{x\in A}{\uparrow}x \quad \hbox{and} \quad \downarrow A=\bigcup_{x\in A}{\downarrow}x.
\end{equation*}
We shall say that a partial map $\alpha\colon X\rightharpoonup X$ is \emph{monotone} if $x\leqslant y$ implies $(x)\alpha\leqslant(y)\alpha$ for $x,y\in \operatorname{dom}\alpha$.

Let $\mathbb{N}$ be the set of positive integers with the usual linear order $\le$. On the Cartesian product $\mathbb{N}\times\mathbb{N}$ we define the product partial order, i.e.,
\begin{equation*}
    (i,m)\leqslant(j,n) \qquad \hbox{if and only if} \qquad (i\leq j) \quad \hbox{and} \quad (m\leq n).
\end{equation*}
Later the set $\mathbb{N}\times\mathbb{N}$ with this partial order will be denoted by $\mathbb{N}^2_{\leqslant}$.

By $\mathscr{P\!O}\!_{\infty}(\mathbb{N}^2_{\leqslant})$ we denote the subsemigroup of injective partial monotone selfmaps of $\mathbb{N}^2_{\leqslant}$ with cofinite domains and images. Obviously, $\mathscr{P\!O}\!_{\infty}(\mathbb{N}^2_{\leqslant})$ is a submonoid of the semigroup $\mathscr{I}_\omega$ and $\mathscr{P\!O}\!_{\infty}(\mathbb{N}^2_{\leqslant})$ is a countable semigroup.

Furthermore, we shall denote the identity of the semigroup $\mathscr{P\!O}\!_{\infty}(\mathbb{N}^2_{\leqslant})$ by $\mathbb{I}$ and the group of units of
$\mathscr{P\!O}\!_{\infty}(\mathbb{N}^2_{\leqslant})$ by $H(\mathbb{I})$.

It well known that each partial injective cofinite selfmap $f$ of $\lambda$ induces a homeomorphism $f^*\colon\lambda^*\rightarrow\lambda^*$ of the remainder $\lambda^*=\beta\lambda\setminus\lambda$ of the Stone-\v{C}ech compactification of the discrete space $\lambda$. Moreover, under some set theoretic axioms (like \textbf{PFA} or \textbf{OCA}), each homeomorphism of $\omega^*$ is induced by some partial injective cofinite selfmap of $\omega$, where $\omega$ is a first infinite cardinal (see  \cite{ShelahSteprans1989}--\cite{Velickovic1993} and the corresponding sections in the book \cite{Weaver-2014}). Thus, the inverse semigroup  $\mathscr{I}^{\mathrm{cf}}_\lambda$ of injective partial selfmaps of an infinite cardinal $\lambda$ with cofinite domains and images admits a natural homomorphism  $\mathfrak{h}\colon \mathscr{I}^{\mathrm{cf}}_\lambda\rightarrow \mathscr{H}(\lambda^*)$ to the homeomorphism group $\mathscr{H}(\lambda^*)$ of $\lambda^*$ and this homomorphism is surjective under certain set theoretic assumptions.

In the paper \cite{Gutik-Repovs-2015}  algebraic properties of the semigroup
$\mathscr{I}^{\mathrm{cf}}_\lambda$ are studied. It is shown that
$\mathscr{I}^{\mathrm{cf}}_\lambda$ is a bisimple inverse semigroup
and that for every non-empty chain $L$ in
$E(\mathscr{I}^{\mathrm{cf}}_\lambda)$ there exists an inverse
subsemigroup $S$ of $\mathscr{I}^{\mathrm{cf}}_\lambda$ such that
$S$ is isomorphic to the bicyclic semigroup and $L\subseteq E(S)$,
described the Green relations on $\mathscr{I}^{\mathrm{cf}}_\lambda$
and proved that every non-trivial congruence on
$\mathscr{I}^{\mathrm{cf}}_\lambda$ is a group congruence. Also, the structure of the quotient semigroup $\mathscr{I}^{\mathrm{cf}}_\lambda/\sigma$, where $\sigma$ is the least group congruence on $\mathscr{I}^{\mathrm{cf}}_\lambda$, is described.

The semigroups $\mathscr{I}_{\infty}^{\!\nearrow}(\mathbb{N})$ and $\mathscr{I}_{\infty}^{\!\nearrow}(\mathbb{Z})$ of injective isotone partial selfmaps with cofinite domains and images of positive integers and integers, respectively, are studied in \cite{Gutik-Repovs-2011} and \cite{Gutik-Repovs-2012}. It was proved that the semigroups $\mathscr{I}_{\infty}^{\!\nearrow}(\mathbb{N})$ and $\mathscr{I}_{\infty}^{\!\nearrow}(\mathbb{Z})$ have similar properties to the bicyclic semigroup: they are bisimple and  every non-trivial homomorphic image $\mathscr{I}_{\infty}^{\!\nearrow}(\mathbb{N})$ and $\mathscr{I}_{\infty}^{\!\nearrow}(\mathbb{Z})$ is a group, and moreover the semigroup $\mathscr{I}_{\infty}^{\!\nearrow}(\mathbb{N})$ has $\mathbb{Z}(+)$ as a maximal group image and $\mathscr{I}_{\infty}^{\!\nearrow}(\mathbb{Z})$ has $\mathbb{Z}(+)\times\mathbb{Z}(+)$, respectively.

In the paper \cite{Gutik-Pozdnyakova-2014} we studied the semigroup
$\mathscr{I\!O}\!_{\infty}(\mathbb{Z}^n_{\operatorname{lex}})$ of monotone injective partial selfmaps of the set of $L_n\times_{\operatorname{lex}}\mathbb{Z}$ having cofinite domain and image, where $L_n\times_{\operatorname{lex}}\mathbb{Z}$ is the lexicographic product of $n$-elements chain and the set of integers with the usual linear order. We described the Green relations on $\mathscr{I\!O}\!_{\infty}(\mathbb{Z}^n_{\operatorname{lex}})$,
showed that the semigroup $\mathscr{I\!O}\!_{\infty}(\mathbb{Z}^n_{\operatorname{lex}})$ is
bisimple and established its projective congruences. Also, we proved that $\mathscr{I\!O}\!_{\infty}(\mathbb{Z}^n_{\operatorname{lex}})$ is finitely generated, every automorphism of $\mathscr{I\!O}\!_{\infty}(\mathbb{Z})$ is inner, and showed that in the case $n\geqslant 2$ the semigroup $\mathscr{I\!O}\!_{\infty}(\mathbb{Z}^n_{\operatorname{lex}})$ has non-inner automorphisms. In \cite{Gutik-Pozdnyakova-2014} we proved that for every positive integer $n$ the quotient semigroup
$\mathscr{I\!O}\!_{\infty}(\mathbb{Z}^n_{\operatorname{lex}})/\sigma$, where $\sigma$ is a least group congruence on $\mathscr{I\!O}\!_{\infty}(\mathbb{Z}^n_{\operatorname{lex}})$, is isomorphic to the direct power $\left(\mathbb{Z}(+)\right)^{2n}$. The structure of the sublattice of congruences on $\mathscr{I\!O}\!_{\infty}(\mathbb{Z}^n_{\operatorname{lex}})$ which are contained in the least group congruence is described in \cite{Gutik-Pozdniakova-2014}.

In this paper we study algebraic properties of the semigroup $\mathscr{P\!O}\!_{\infty}(\mathbb{N}^2_{\leqslant})$. We describe properties of elements of the semigroup $\mathscr{P\!O}\!_{\infty}(\mathbb{N}^2_{\leqslant})$ as monotone partial bijection of $\mathbb{N}^{2}_{\leqslant}$ and show that the group of units of $\mathscr{P\!O}\!_{\infty}(\mathbb{N}^2_{\leqslant})$ is isomorphic to the cyclic group of the order two. Also, the subsemigroup of idempotents of $\mathscr{P\!O}\!_{\infty}(\mathbb{N}^2_{\leqslant})$ and the Green relations on $\mathscr{P\!O}\!_{\infty}(\mathbb{N}^2_{\leqslant})$ are described. In particular, we show that $\mathscr{D}=\mathscr{J}$ in $\mathscr{P\!O}\!_{\infty}(\mathbb{N}^2_{\leqslant})$.


\section{Properties of elements of the semigroup
$\mathscr{P\!O}\!_{\infty}(\mathbb{N}^2_{\leqslant})$ as monotone partial permutations}\label{section-2}

In this short section we describe properties of elements of the semigroup
$\mathscr{P\!O}\!_{\infty}(\mathbb{N}^2_{\leqslant})$ as monotone partial transformations of the poset $\mathbb{N}^2_{\leqslant}$.

For any $n\in\mathbb{N}$ and an arbitrary $\alpha\in \mathscr{P\!O}\!_{\infty}(\mathbb{N}^2_{\leqslant})$ we denote:
\begin{equation*}
\begin{split}
  \textsf{V}^n =\{(n,j)\colon j\in\mathbb{N}\}; \qquad & \qquad \textsf{H}^n =\{(j,n)\colon j\in\mathbb{N}\};\\
  \textsf{V}^n_{\operatorname{dom}\alpha}=\textsf{V}^n\cap\operatorname{dom}\alpha;   \qquad & \qquad
  \textsf{V}^n_{\operatorname{ran}\alpha}=\textsf{V}^n\cap\operatorname{ran}\alpha;\\
  \textsf{H}^n_{\operatorname{dom}\alpha}=\textsf{H}^n\cap\operatorname{dom}\alpha;   \qquad & \qquad
  \textsf{H}^n_{\operatorname{ran}\alpha}=\textsf{H}^n\cap\operatorname{ran}\alpha.
\end{split}
\end{equation*}

\begin{remark}\label{remark-2.1}
We observe that the definition of the semigroup $\mathscr{P\!O}\!_{\infty}(\mathbb{N}^2_{\leqslant})$ implies that for any $n\in\mathbb{N}$ and arbitrary $\alpha\in \mathscr{P\!O}\!_{\infty}(\mathbb{N}^2_{\leqslant})$ the sets $\textsf{V}^n_{\operatorname{dom}\alpha}$, $\textsf{V}^n_{\operatorname{ran}\alpha}$, $\textsf{H}^n_{\operatorname{dom}\alpha}$ and $\textsf{H}^n_{\operatorname{ran}\alpha}$ are infinite, and moreover all of these sets with the partial order induced from $\mathbb{N}^2_{\leqslant}$ are order isomorphic to $(\mathbb{N},\leq)$.
\end{remark}

\begin{lemma}\label{lemma-2.2}
There exists no element $\alpha$ of the semigroup $\mathscr{P\!O}\!_{\infty}(\mathbb{N}^2_{\leqslant})$ such that $(m,n)<(m,n)\alpha$ for some $(m,n)\in\operatorname{dom}\alpha$.
\end{lemma}

\begin{proof}
Suppose the contrary, i.e., that there exists an element $\alpha$ of the semigroup $\mathscr{P\!O}\!_{\infty}(\mathbb{N}^2_{\leqslant})$ such that $(m,n)<(m,n)\alpha$ for some $(m,n)\in\operatorname{dom}\alpha$. We denote $(m,n)\alpha=(i,j)$. Then our assumption implies that the family of subsets
\begin{equation*}
\mathfrak{R}_\alpha=\left\{\textsf{V}^k_{\operatorname{ran}\alpha}\colon k<i\right\}\cup \left\{\textsf{H}^k_{\operatorname{ran}\alpha}\colon k<j\right\}
\end{equation*}
has more elements than the family
\begin{equation*}
\mathfrak{D}_\alpha=\left\{\textsf{V}^k_{\operatorname{dom}\alpha}\colon k<m\right\}\cup \left\{\textsf{H}^k_{\operatorname{dom}\alpha}\colon k<n\right\}.
\end{equation*}
Then there exist $A\in\mathfrak{D}_\alpha$ and distinct $B_1,B_2\in\mathfrak{R}_\alpha$ such that the following conditions hold:
\begin{itemize}
  \item[$(i)$] $(p,q)\alpha\in B_1$ for infinitely many $(p,q)\in A$; \; and
  \item[$(ii)$] $(s,t)\alpha\in B_2$ for infinitely many $(s,t)\in A$.
\end{itemize}
We observe that $A$ is a linearly ordered subset of the poset $\mathbb{N}^2_{\leqslant}$. Hence, the definition of the semigroup $\mathscr{P\!O}\!_{\infty}(\mathbb{N}^2_{\leqslant})$ implies that the image $(A)\alpha$ must be a linearly ordered subset of the poset $\mathbb{N}^2_{\leqslant}$ as well. This implies that one of the following conditions holds:
\begin{itemize}
  \item[$(a)$] there exist distinct elements $\textsf{V}^{k_1}_{\operatorname{ran}\alpha}$ and $\textsf{V}^{k_2}_{\operatorname{ran}\alpha}$ of the family $\mathfrak{R}_\alpha$ such that the sets $\textsf{V}^{k_1}_{\operatorname{ran}\alpha}\cap(A)\alpha$ and $\textsf{V}^{k_2}_{\operatorname{ran}\alpha}\cap(A)\alpha$ are infinite;
  \item[$(b)$] there exist distinct elements $\textsf{H}^{k_1}_{\operatorname{ran}\alpha}$ and $\textsf{H}^{k_2}_{\operatorname{ran}\alpha}$ of the family $\mathfrak{R}_\alpha$ such that the sets $\textsf{H}^{k_1}_{\operatorname{ran}\alpha}\cap(A)\alpha$ and $\textsf{H}^{k_2}_{\operatorname{ran}\alpha}\cap(A)\alpha$ are infinite;
  \item[$(c)$] there exist distinct elements $\textsf{V}^{k_1}_{\operatorname{ran}\alpha}$ and $\textsf{H}^{k_2}_{\operatorname{ran}\alpha}$ of the family $\mathfrak{R}_\alpha$ such that the sets $\textsf{V}^{k_1}_{\operatorname{ran}\alpha}\cap(A)\alpha$ and $\textsf{H}^{k_2}_{\operatorname{ran}\alpha}\cap(A)\alpha$ are infinite.
\end{itemize}
Each of the above conditions contradicts the fact that $(A)\alpha$ is a linearly ordered subset of the poset $\mathbb{N}^2_{\leqslant}$. The obtained contradiction implies the statement of the lemma.
\end{proof}

By $\varpi$ we denote the bijective transformation of $\mathbb{N}\times\mathbb{N}$ defined by the formula $(i,j)\varpi=(j,i)$, for any $(i,j)\in\mathbb{N}\times\mathbb{N}$. It is obvious that $\varpi$ is an element of the semigroup $\mathscr{P\!O}\!_{\infty}(\mathbb{N}^2_{\leqslant})$ and $\varpi\varpi=\mathbb{I}$.

\begin{lemma}\label{lemma-2.3}
There exists no element $\alpha$ of the semigroup $\mathscr{P\!O}\!_{\infty}(\mathbb{N}^2_{\leqslant})$ such that $(n,m)<(m,n)\alpha$ for some $(m,n)\in\operatorname{dom}\alpha$.
\end{lemma}

\begin{proof}
Suppose the contrary. Then there exists an element $\alpha$ of the semigroup $\mathscr{P\!O}\!_{\infty}(\mathbb{N}^2_{\leqslant})$ such that $(n,m)<(m,n)\alpha$ for some $(m,n)\in\operatorname{dom}\alpha$. Then we obtain that $(m,n)<(m,n)\alpha\varpi$, which contradicts Lemma~\ref{lemma-2.2}. The obtained contradiction implies the statement of our lemma.
\end{proof}

For arbitrary positive integer $l$ we define a partial map $\alpha_{\textsf{V}}^l\colon \mathbb{N}^2\rightharpoonup\mathbb{N}^2$ in the following way:
\begin{equation*}
\begin{split}
  \operatorname{dom}(\alpha_{\textsf{V}}^l)= & \; \mathbb{N}^2\setminus\left\{(1,1), \ldots,(l,1)\right\}, \qquad \operatorname{ran}(\alpha_{\textsf{V}}^l)=\mathbb{N}^2 \qquad \hbox{and} \\
    & (i,j)\alpha_{\textsf{V}}^l=
    \left\{
      \begin{array}{ll}
        (i,j),   & \hbox{if~} i>l;\\
        (i,j-1), & \hbox{if~} i\leq l.
      \end{array}
    \right.
\end{split}
\end{equation*}
It it obvious that $\alpha_{\textsf{V}}^l\in\mathscr{P\!O}\!_{\infty}(\mathbb{N}^2_{\leqslant})$ for any positive integer $l$.

\begin{lemma}\label{lemma-2.4}
For any element $\alpha$ of the semigroup $\mathscr{P\!O}\!_{\infty}(\mathbb{N}^2_{\leqslant})$ the following assertions hold:
\begin{itemize}
  \item[$(1)$] either \emph{$(\textsf{H}^{1}_{\operatorname{dom}\alpha})\alpha\subseteq \textsf{H}^1$} or \emph{$(\textsf{H}^1_{\operatorname{dom}\alpha})\alpha\subseteq \textsf{V}^1$};
  \item[$(2)$] either \emph{$(\textsf{V}^{1}_{\operatorname{dom}\alpha})\alpha\subseteq \textsf{V}^1$} or \emph{$(\textsf{V}^1_{\operatorname{dom}\alpha})\alpha\subseteq \textsf{H}^1$}.
\end{itemize}
\end{lemma}

\begin{proof}
We shall show that assertion $(1)$ holds. The proof of $(2)$ is similar.

First we observe that $(\textsf{H}^{1}_{\operatorname{dom}\alpha})\alpha\subseteq \textsf{H}^1$ if and only if $(\textsf{H}^{1}_{\operatorname{dom}\alpha})\alpha\varpi\subseteq \textsf{V}^1$.

Suppose the contrary: there exists an element $\alpha$ of the semigroup $\mathscr{P\!O}\!_{\infty}(\mathbb{N}^2_{\leqslant})$ such that neither $(\textsf{H}^{1}_{\operatorname{dom}\alpha})\alpha\subseteq \textsf{H}^1$ nor $(\textsf{H}^1_{\operatorname{dom}\alpha})\alpha\subseteq \textsf{V}^1$. Then the definition of the semigroup $\mathscr{P\!O}\!_{\infty}(\mathbb{N}^2_{\leqslant})$, Lemma~\ref{lemma-2.2} and the above observation imply that without loss of generality we may assume that $(\textsf{H}^{1}_{\operatorname{dom}\alpha})\alpha\nsubseteq \textsf{H}^1\cup \textsf{V}^1$ and there exists $(k,1)\in\operatorname{dom}\alpha$ such that $(k,1)\alpha=(i,j)$, $j\neq 1$ and $2\leq i<k$. Also, by the definition of $\alpha_{\textsf{V}}^l\in\mathscr{P\!O}\!_{\infty}(\mathbb{N}^2_{\leqslant})$ we get that without loss of generality we may assume that $j=2$, i.e., $(k,1)\alpha=(i,2)$. Then there exist disjoint infinite subsets $A$ and $B$ of the set $\textsf{V}^1_{\operatorname{dom}\alpha}\cup \ldots \cup \textsf{V}^{k-1}_{\operatorname{dom}\alpha}$ such that
\begin{equation*}
    A\cup B=\textsf{V}^1_{\operatorname{dom}\alpha}\cup \ldots \cup \textsf{V}^{k-1}_{\operatorname{dom}\alpha}, \qquad
    \textsf{H}^{1}_{\operatorname{ran}\alpha}\subseteq (A)\alpha \qquad \hbox{and} \qquad
    \textsf{V}^{1}_{\operatorname{ran}\alpha}\cup \ldots\cup\textsf{V}^{k-1}_{\operatorname{ran}\alpha}\subseteq (B)\alpha.
\end{equation*}

If $A\cap\textsf{V}^1_{\operatorname{dom}\alpha}\neq\varnothing$ then the definition of the semigroup $\mathscr{P\!O}\!_{\infty}(\mathbb{N}^2_{\leqslant})$ and Lemma~\ref{lemma-2.2} imply that there exists $(a,b)\in B$ such that $(a,b)\alpha\in\textsf{V}^{1}_{\operatorname{ran}\alpha}$ and $(c,d)\leqslant(a,b)$ for some $(c,d)\in A$, which contradicts the definition of the partial order $\leqslant$ of the poset $\mathbb{N}^2_{\leqslant}$.

Assume that $A\subseteq\textsf{V}^2_{\operatorname{dom}\alpha}\cup\ldots\cup \textsf{V}^{k-1}_{\operatorname{dom}\alpha}$. Then there exist infinite subsets $A_1\subseteq A$ and $B_1\subseteq B$ such that $(A_1)\alpha=\textsf{H}^{1}_{\operatorname{ran}\alpha}\setminus\left\{(1,1)\right\}$ and $(B_1)\alpha=\textsf{V}^{1}_{\operatorname{ran}\alpha}\setminus\left\{(1,1)\right\}$. Hence the definition of the poset $\mathbb{N}^2_{\leqslant}$ implies that at least one of the following conditions holds: ${\uparrow}A_1\cap{\downarrow}B_1\neq\varnothing$ or ${\downarrow}A_1\cap{\uparrow}B_1\neq\varnothing$. If ${\uparrow}A_1\cap{\downarrow}B_1\neq\varnothing$ then $({\downarrow}B_1)\alpha\subseteq{\downarrow}\textsf{V}^{1}_{\operatorname{ran}\alpha}=\textsf{V}^{1}$ but $\textsf{V}^{1}\cap{\uparrow}\left(\textsf{H}^{1}_{\operatorname{ran}\alpha}\setminus\left\{(1,1)\right\}\right)\subseteq \textsf{V}^{1}\cap{\uparrow}\left(\textsf{H}^{1}\setminus\left\{(1,1)\right\}\right)=\varnothing$, a contradiction. Similarly, if ${\downarrow}A_1\cap{\uparrow}B_1\neq\varnothing$ then $({\downarrow}A_1)\alpha\subseteq{\downarrow}\textsf{H}^{1}_{\operatorname{ran}\alpha}=\textsf{H}^{1}$ and we get a contradiction with
\begin{equation*}
\textsf{H}^{1}\cap{\uparrow}\left(\textsf{V}^{1}_{\operatorname{ran}\alpha}\setminus\left\{(1,1)\right\}\right)\subseteq \textsf{H}^{1}\cap{\uparrow}\left(\textsf{V}^{1}\setminus\left\{(1,1)\right\}\right)=\varnothing.
\end{equation*}

The obtained contradictions imply the statement of the lemma.
\end{proof}

\begin{proposition}\label{proposition-2.5}
Let $\alpha$ be an arbitrary element of the semigroup $\mathscr{P\!O}\!_{\infty}(\mathbb{N}^2_{\leqslant})$. Then the following assertions hold:
\begin{itemize}
  \item[$(1)$] \emph{$(\textsf{H}^{1}_{\operatorname{dom}\alpha})\alpha\subseteq \textsf{H}^1$} if and only if \emph{$(\textsf{V}^{1}_{\operatorname{dom}\alpha})\alpha\subseteq \textsf{V}^1$}, and moreover in this case the sets \emph{$\textsf{H}^1\setminus(\textsf{H}^{1}_{\operatorname{dom}\alpha})\alpha$} and \emph{$\textsf{V}^1\setminus(\textsf{V}^{1}_{\operatorname{dom}\alpha})\alpha$} are finite;
  \item[$(2)$] \emph{$(\textsf{H}^1_{\operatorname{dom}\alpha})\alpha\subseteq \textsf{V}^1$} if and only if \emph{$(\textsf{V}^1_{\operatorname{dom}\alpha})\alpha\subseteq \textsf{H}^1$}, and moreover in this case \emph{$\textsf{V}^1\setminus(\textsf{H}^{1}_{\operatorname{dom}\alpha})\alpha$} and \emph{$\textsf{H}^1\setminus(\textsf{V}^{1}_{\operatorname{dom}\alpha})\alpha$} are finite.
\end{itemize}
\end{proposition}

\begin{proof}
The first statements of assertions $(1)$ and $(2)$ follow from Lemma~\ref{lemma-2.4} and their second parts follow from Lemma~\ref{lemma-2.2}.
\end{proof}

\begin{theorem}\label{theorem-2.6}
Let $\alpha$ be an arbitrary element of the semigroup $\mathscr{P\!O}\!_{\infty}(\mathbb{N}^2_{\leqslant})$ and $n$ be an arbitrary positive integer. Then the following assertions hold:
\begin{itemize}
  \item[$(1)$] if \emph{$(\textsf{H}^{1}_{\operatorname{dom}\alpha})\alpha\subseteq \textsf{H}^1$} then \emph{$(\textsf{H}^{n}_{\operatorname{dom}\alpha})\alpha{\subseteq^*} \textsf{H}^n$} and \emph{$(\textsf{V}^{n}_{\operatorname{dom}\alpha})\alpha{\subseteq^*} \textsf{V}^n$},  and moreover
      \begin{equation*}
      \emph{$(\textsf{H}^{1}_{\operatorname{dom}\alpha}\cup\ldots\cup\textsf{H}^{n}_{\operatorname{dom}\alpha})\alpha\subseteq
      \textsf{H}^1\cup\ldots\cup\textsf{H}^n$}  \hbox{~and~}
      \emph{$(\textsf{V}^{1}_{\operatorname{dom}\alpha}\cup\ldots\cup\textsf{V}^{n}_{\operatorname{dom}\alpha})\alpha\subseteq
      \textsf{V}^1\cup\ldots\cup\textsf{V}^n$};
      \end{equation*}
  \item[$(2)$] if \emph{$(\textsf{H}^1_{\operatorname{dom}\alpha})\alpha\subseteq \textsf{V}^1$} then \emph{$(\textsf{H}^{n}_{\operatorname{dom}\alpha})\alpha{\subseteq^*} \textsf{V}^n$} and \emph{$(\textsf{V}^{n}_{\operatorname{dom}\alpha})\alpha{\subseteq^*} \textsf{H}^n$},  and moreover
      \begin{equation*}
      \emph{$(\textsf{H}^{1}_{\operatorname{dom}\alpha}\cup\ldots\cup\textsf{H}^{n}_{\operatorname{dom}\alpha})\alpha\subseteq
      \textsf{V}^1\cup\ldots\cup\textsf{V}^n$} \hbox{~and~}
      \emph{$(\textsf{V}^{1}_{\operatorname{dom}\alpha}\cup\ldots\cup\textsf{V}^{n}_{\operatorname{dom}\alpha})\alpha\subseteq
      \textsf{H}^1\cup\ldots\cup\textsf{H}^n$}.
      \end{equation*}
\end{itemize}
\end{theorem}

\begin{proof}
$(1)$ We shall prove this assertion by induction.

In the case when $n=1$ our statement follows from Lemma~\ref{lemma-2.4} and Proposition~\ref{proposition-2.5}. Next we shall show that the step of induction holds.

We assume that our assertion holds for arbitrary $\alpha\in\mathscr{P\!O}\!_{\infty}(\mathbb{N}^2_{\leqslant})$ and for all positive integers $n\le k$ and we shall prove that then the assertion is true in the case when $n=k+1$.

For an arbitrary element $\alpha$ of the semigroup $\mathscr{P\!O}\!_{\infty}(\mathbb{N}^2_{\leqslant})$ we define a partial map $\alpha_{[k+1]}\colon \mathbb{N}^2\rightharpoonup\mathbb{N}^2$ in the following way:
\begin{equation*}
\begin{split}
  (i,j)\alpha_{[k+1]} & \quad \hbox{is defined if and only if} \quad (i,j)\in\operatorname{dom}\alpha\cap{\uparrow}(k+1,k+1)  \\
                      & \quad \hbox{and} \quad (i,j)\alpha\in\operatorname{ran}\alpha\cap{\uparrow}(k+1,k+1),
                                    \quad \hbox{and moreover in this case we put}\\
                      & \quad (i,j)\alpha_{[k+1]}=(i,j)\alpha,
\end{split}
\end{equation*}
i.e., the partial map $\alpha_{[k+1]}\colon \mathbb{N}^2\rightharpoonup\mathbb{N}^2$ is the restriction of the partial map $\alpha\colon \mathbb{N}^2\rightharpoonup\mathbb{N}^2$ onto the set ${\uparrow}(k+1,k+1)$. Since the set ${\uparrow}(k+1,k+1)$ with the partial induced from $\mathbb{N}^2_{\leqslant}$ is order isomorphic to $\mathbb{N}^2_{\leqslant}$, the assumption of induction and Lemma~\ref{lemma-2.4} imply that
either $(\textsf{H}^{k+1}\cap\operatorname{dom}(\alpha_{[k+1]}))\alpha_{[k+1]}\subseteq \textsf{H}^{k+1}$ or $(\textsf{H}^{k+1}\cap\operatorname{dom}(\alpha_{[k+1]}))\alpha_{[k+1]}\subseteq \textsf{V}^{k+1}$. Then the inclusion
\begin{equation*}
{\downarrow}(\textsf{H}^{1}_{\operatorname{dom}\alpha}\cup\ldots\cup\textsf{H}^{k}_{\operatorname{dom}\alpha})\subseteq {\downarrow}(\textsf{H}^{1}_{\operatorname{dom}\alpha}\cup\ldots\cup\textsf{H}^{k}_{\operatorname{dom}\alpha}\cup\textsf{H}^{k+1}_{\operatorname{dom}\alpha})
\end{equation*}
implies that
\begin{equation*}
    (\textsf{H}^{k+1}\cap\operatorname{dom}(\alpha_{[k+1]}))\alpha=(\textsf{H}^{k+1}\cap\operatorname{dom}(\alpha_{[k+1]}))\alpha_{[k+1]}\subseteq \textsf{H}^{k+1}.
\end{equation*}
Hence we have that $(\textsf{H}^{k+1}_{\operatorname{dom}\alpha})\alpha{\subseteq^*} \textsf{H}^{k+1}$, because the set $\operatorname{dom}\alpha\setminus\operatorname{dom}(\alpha_{[k+1]})\cap\textsf{H}^{k+1}$ is finite. Also, since $(i,j)\leqslant(p,q)$ for all $(i,j)\in \operatorname{dom}\alpha\setminus\operatorname{dom}(\alpha_{[k+1]})\cap\textsf{H}^{k+1}$ and $(p,q)\in\operatorname{dom}(\alpha_{[k+1]})\cap\textsf{H}^{k+1}$, the definition of the semigroup $\mathscr{P\!O}\!_{\infty}(\mathbb{N}^2_{\leqslant})$, the assumption of induction and the inclusion $(\textsf{H}^{k+1}\cap\operatorname{dom}(\alpha_{[k+1]}))\alpha\subseteq \textsf{H}^{k+1}$ imply the requested inclusion
\begin{equation*}
    (\textsf{H}^{1}_{\operatorname{dom}\alpha}\cup\ldots\cup\textsf{H}^{k}_{\operatorname{dom}\alpha}\cup\textsf{H}^{k+1}_{\operatorname{dom}\alpha})\alpha\subseteq
      \textsf{H}^1\cup\ldots\cup\textsf{H}^{k}\cup\textsf{H}^{k+1}.
\end{equation*}

Again using indiction and Proposition~\ref{proposition-2.5} we get that the condition $(\textsf{H}^{1}_{\operatorname{dom}\alpha})\alpha\subseteq \textsf{H}^1$ implies that  $(\textsf{H}^{n}_{\operatorname{dom}\alpha})\alpha{\subseteq^*} \textsf{H}^n$ and $(\textsf{V}^{1}_{\operatorname{dom}\alpha}\cup\ldots\cup\textsf{V}^{n}_{\operatorname{dom}\alpha})\alpha\subseteq \textsf{V}^1\cup\ldots\cup\textsf{V}^n$ for every positive integer $n$.

$(2)$ If $(\textsf{H}^1_{\operatorname{dom}\alpha})\alpha\subseteq \textsf{V}^1$ then $(\textsf{H}^1_{\operatorname{dom}\alpha})\alpha\varpi\subseteq \textsf{H}^1$. Then assertion $(1)$ and the equality $\alpha\varpi\varpi=\alpha$ imply assertion $(2)$.
\end{proof}

The following theorem describes the structure of elements of the  semigroup $\mathscr{P\!O}\!_{\infty}(\mathbb{N}^2_{\leqslant})$ as monotone partial permutations of the poset $\mathbb{N}^2_{\leqslant}$.

\begin{theorem}\label{theorem-2.7}
Let $\alpha$ be an arbitrary element of the semigroup $\mathscr{P\!O}\!_{\infty}(\mathbb{N}^2_{\leqslant})$. Then the following assertions hold:
\begin{itemize}
  \item[$(1)$] if \emph{$(\textsf{H}^{1}_{\operatorname{dom}\alpha})\alpha\subseteq \textsf{H}^1$} then
     \begin{itemize}
       \item[$(i_1)$] $(i,j)\alpha\leqslant(i,j)$ for each $(i,j)\in\operatorname{dom}\alpha$; \: and
       \item[$(ii_1)$] there exists a smallest positive integer $n_{\alpha}$ such that $(i,j)\alpha=(i,j)$ for each $(i,j)\in\operatorname{dom}\alpha\cap{\uparrow}(n_{\alpha},n_{\alpha})$;
     \end{itemize}
  \item[$(2)$] if \emph{$(\textsf{H}^1_{\operatorname{dom}\alpha})\alpha\subseteq \textsf{V}^1$} then
     \begin{itemize}
       \item[$(i_2)$] $(i,j)\alpha\leqslant(j,i)$ for each $(i,j)\in\operatorname{dom}\alpha$; \: and
       \item[$(ii_2)$] there exists a smallest positive integer $n_{\alpha}$ such that $(i,j)\alpha=(j,i)$ for each $(i,j)\in\operatorname{dom}\alpha\cap{\uparrow}(n_{\alpha},n_{\alpha})$.
     \end{itemize}
\end{itemize}
\end{theorem}

\begin{proof}
$(1)$ Fix an arbitrary element $\alpha$ of the semigroup $\mathscr{P\!O}\!_{\infty}(\mathbb{N}^2_{\leqslant})$ such that $(\textsf{H}^{1}_{\operatorname{dom}\alpha})\alpha\subseteq \textsf{H}^1$. Suppose to the contrary that there exists $(i,j)\in\operatorname{dom}\alpha$ such that $(i,j)\alpha=(k,l)\nleqslant(i,j)$. Then Lemma~\ref{lemma-2.2}, Theorem~\ref{theorem-2.6}$(1)$ and the definition of the partial order of the poset $\mathbb{N}^2_{\leqslant}$ imply that $k>i$ and $l<j$. Now, by the definition of the semigroup $\mathscr{P\!O}\!_{\infty}(\mathbb{N}^2_{\leqslant})$ we get that there exists a positive integer $m\leq i$ such that
\begin{equation*}
(\textsf{V}^{1}_{\operatorname{dom}\alpha}\cup\ldots\cup\textsf{V}^{m}_{\operatorname{dom}\alpha})\alpha\nsubseteq
\textsf{V}^1\cup\ldots\cup\textsf{V}^m,
\end{equation*}
which contradicts Theorem~\ref{theorem-2.6}$(1)$. The obtained contradiction implies the requested inequality $(i,j)\alpha\leqslant(i,j)$ and this completes the proof of $(i)$.

Next we shall prove $(ii)$. Fix an arbitrary element $\alpha$ of the semigroup $\mathscr{P\!O}\!_{\infty}(\mathbb{N}^2_{\leqslant})$ such that $(\textsf{H}^{1}_{\operatorname{dom}\alpha})\alpha\subseteq \textsf{H}^1$. Suppose to the contrary that for any positive integer $n$ there exists $(i,j)\in\operatorname{dom}\alpha\cap{\uparrow}(n,n)$ such that $(i,j)\alpha\neq(i,j)$. We put $\textsf{N}_{\operatorname{dom}\alpha}=\left|\mathbb{N}^2\setminus\operatorname{dom}\alpha\right|+1$ and
\begin{equation*}
 \textsf{M}_{\operatorname{dom}\alpha}=\max\left\{\left\{i\colon(i,j)\notin\operatorname{dom}\alpha\right\}, \left\{j\colon(i,j)\notin\operatorname{dom}\alpha\right\}\right\}+1.
\end{equation*}
The definition of the semigroup $\mathscr{P\!O}\!_{\infty}(\mathbb{N}^2_{\leqslant})$ implies that the positive integers $\textsf{N}_{\operatorname{dom}\alpha}$ and $\textsf{M}_{\operatorname{dom}\alpha}$ are well defined. Put $n_0=\max\left\{\textsf{N}_{\operatorname{dom}\alpha},\textsf{M}_{\operatorname{dom}\alpha}\right\}$. Then our assumption implies that there exists $(i,j)\in\operatorname{dom}\alpha\cap{\uparrow}(n_0,n_0)$ such that $(i,j)\alpha=(i_{\alpha},j_{\alpha})\neq(i,j)$. By $(i)$, we have that $(i_{\alpha},j_{\alpha})<(i,j)$. We consider the case when $i_{\alpha}<i$. In the case when $j_{\alpha}<j$ the proof is similar. Assume that $i\leq j$. By Theorem~\ref{theorem-2.6} the partial bijection $\alpha$ maps the set $S_i=\left\{(n,m)\colon n ,m\leq i-1\right\}$ into itself. Also, by the definition of the semigroup $\mathscr{P\!O}\!_{\infty}(\mathbb{N}^2_{\leqslant})$ the partial bijection $\alpha$ maps the set $\left\{(i,1), \ldots, (i,i)\right\}$ into $S_i$ as well. Then our construction implies that
\begin{equation*}
\left|S_i\setminus\operatorname{dom}\alpha\right|=\left|\mathbb{N}^2\setminus\operatorname{dom}\alpha\right|=\textsf{N}_{\operatorname{dom}\alpha}-1 \qquad \hbox{and} \qquad \left|\left\{(i,1), \ldots, (i,i)\right\}\right|\geq \textsf{N}_{\operatorname{dom}\alpha},
\end{equation*}
a contradiction. In the case when $j\leq i$ we get a contradiction in a similar way. This completes the proof of existence of such a positive integer $n_{\alpha}$ for any $\alpha\in\mathscr{P\!O}\!_{\infty}(\mathbb{N}^2_{\leqslant})$. The existence of such minimal positive integer $n_{\alpha}$ follows from the fact that the set of all positive integers with the usual order $\leq$ is well-ordered.

$(2)$ If $(\textsf{H}^1_{\operatorname{dom}\alpha})\alpha\subseteq \textsf{V}^1$ then $(\textsf{H}^1_{\operatorname{dom}\alpha})\alpha\varpi\subseteq \textsf{H}^1$, and hence $(1)$ and the equality $\alpha\varpi\varpi=\alpha$ imply our assertion.
\end{proof}

Theorem~\ref{theorem-2.7} implies the following corollary:

\begin{corollary}\label{corollary-2.8}
$\left|\mathbb{N}^2\setminus\operatorname{ran}\alpha\right|\leq\left|\mathbb{N}^2\setminus\operatorname{dom}\alpha\right|$ for an arbitrary $\alpha\in\mathscr{P\!O}\!_{\infty}(\mathbb{N}^2_{\leqslant})$.
\end{corollary}

For an arbitrary non-empty subset $A$ of $\mathbb{N}\times\mathbb{N}$ and any element $(i,j)\in\mathbb{N}\times\mathbb{N}$ we denote $\overline{A}=\left\{(i,j)\in\mathbb{N}\times\mathbb{N}\colon (j,i)\in A\right\}$ and $\overline{(i,j)}=(j,i)$.

\begin{proposition}\label{proposition-2.9}
Let $\alpha$ be an arbitrary element of the semigroup $\mathscr{P\!O}\!_{\infty}(\mathbb{N}^2_{\leqslant})$. Then the following as\-ser\-tions hold:
\begin{itemize}
  \item[$(i)$] $\operatorname{dom}(\varpi\alpha)=\operatorname{dom}(\varpi\alpha\varpi)=\overline{\operatorname{dom}\alpha}$ and $\operatorname{dom}(\alpha\varpi)=\operatorname{dom}\alpha$;
  \item[$(ii)$] $\operatorname{ran}(\varpi\alpha)=\operatorname{ran}\alpha$ and $\operatorname{ran}(\varpi\alpha\varpi)=\operatorname{ran}(\alpha\varpi)=\overline{\operatorname{ran}\alpha}$;
  \item[$(iii)$] $\alpha$ is an idempotent if and only if so is $\varpi\alpha\varpi$.
\end{itemize}
\end{proposition}

\begin{proof}
Items $(i)$ and $(ii)$ follow from the definition of the composition of partial maps.

$(iii)$ Suppose that $\alpha$ is an idempotent of the semigroup $\mathscr{P\!O}\!_{\infty}(\mathbb{N}^2_{\leqslant})$. By items $(i)$ and $(ii)$ we have that $\operatorname{dom}(\varpi\alpha\varpi)=\overline{\operatorname{dom}\alpha}=\overline{\operatorname{ran}\alpha}=\operatorname{ran}(\varpi\alpha\varpi)$. Then
$(j,i)\varpi\alpha\varpi=(i,j)\alpha\varpi=(i,j)\varpi=(j,i)$ for an arbitrary $(i,j)\in\operatorname{dom}\alpha$, and hence $\varpi\alpha\varpi\in E(\mathscr{P\!O}\!_{\infty}(\mathbb{N}^2_{\leqslant}))$. The converse statement follows from the equality $\varpi\varpi=\mathbb{I}$.
\end{proof}

The following statement follows from the definition of the semigroup $\mathscr{P\!O}\!_{\infty}(\mathbb{N}^2_{\leqslant})$ and Lemma~\ref{lemma-2.4}.

\begin{proposition}\label{proposition-2.10}
Let $\alpha$ and $\beta$ be arbitrary elements of the semigroup $\mathscr{P\!O}\!_{\infty}(\mathbb{N}^2_{\leqslant})$. Then  \emph{$(\textsf{H}^{1}_{\operatorname{dom}(\alpha\beta)})\alpha\beta\subseteq \textsf{H}^1$} if and only if \emph{$(\textsf{H}^{1}_{\operatorname{dom}(\beta\alpha)})\beta\alpha\subseteq \textsf{H}^1$}.
\end{proposition}

\section{Algebraic properties of the semigroup
$\mathscr{P\!O}\!_{\infty}(\mathbb{N}^2_{\leqslant})$}\label{section-3}

Theorems~\ref{theorem-2.6} and~\ref{theorem-2.7} imply the following

\begin{proposition}\label{proposition-3.1}
The group of units $H(\mathbb{I})$ of the semigroup
$\mathscr{P\!O}\!_{\infty}(\mathbb{N}^2_{\leqslant})$ is isomorphic to $\mathbb{Z}_2$.
\end{proposition}

\begin{proposition}\label{proposition-3.2}
Let $\alpha$ be an element of the semigroup $\mathscr{P\!O}\!_{\infty}(\mathbb{N}^2_{\leqslant})$. Then $\alpha\in H(\mathbb{I})$ if and only if $\operatorname{dom}\alpha=\mathbb{N}^2$.
\end{proposition}

\begin{proof}
The implication $(\Rightarrow)$ is trivial. The implication $(\Leftarrow)$ follows from Theorems~\ref{theorem-2.6}, \ref{theorem-2.7} and Corollary~\ref{corollary-2.8}.
\end{proof}

\begin{proposition}\label{proposition-3.3}
An element $\alpha$ of $\mathscr{P\!O}\!_{\infty}(\mathbb{N}^2_{\leqslant})$ is an idempotent if and only if $\alpha$ is an identity partial self-map of $\mathbb{N}^2_{\leqslant}$ with the cofinite domain.
\end{proposition}

\begin{proof}
The implication $(\Leftarrow)$ is trivial.

$(\Rightarrow)$ Let an element $\alpha$ be an idempotent of the semigroup $\mathscr{P\!O}\!_{\infty}(\mathbb{N}^2_{\leqslant})$. Then for every $x\in\operatorname{dom}\alpha$ we have that $(x)\alpha\alpha=(x)\alpha$ and hence we get that $\operatorname{dom}\alpha^2=\operatorname{dom}\alpha$ and $\operatorname{ran}\alpha^2=\operatorname{ran}\alpha$. Also since $\alpha$ is a partial bijective self-map of $\mathbb{N}^2_{\leqslant}$ we conclude that the previous equalities imply that $\operatorname{dom}\alpha=\operatorname{ran}\alpha$. Fix an arbitrary $x\in\operatorname{dom}\alpha$ and suppose that $(x)\alpha=y$. Then $(x)\alpha=(x)\alpha\alpha=(y)\alpha=y$. Since $\alpha$ is a partial bijective self-map of $\mathbb{N}^2_{\leqslant}$ we have that the equality $(y)\alpha=y$ implies that the full preimage of $y$ under the partial map $\alpha$ is equal to $y$. Similarly the equality $(x)\alpha=y$ implies that the full preimage of $y$ under the partial map $\alpha$ is equal to $x$. Thus we get that $x=y$ and our implication holds.
\end{proof}

\begin{remark}\label{remark-3.4}
The proof of Proposition~\ref{proposition-3.3} implies that the statement of the proposition holds for any semigroup of partial bijections, but in the general case of a semigroup of transformations this statement is not true.
\end{remark}

The following theorem describes the subset of idempotents of the semigroup $\mathscr{P\!O}\!_{\infty}(\mathbb{N}^2_{\leqslant})$.

\begin{theorem}\label{theorem-3.5-1}
For an element $\alpha$ of the semigroup $\mathscr{P\!O}\!_{\infty}(\mathbb{N}^2_{\leqslant})$ the following conditions are equivalent:
\begin{itemize}
  \item[$(i)$] $\alpha$ is an idempotent of $\mathscr{P\!O}\!_{\infty}(\mathbb{N}^2_{\leqslant})$;
  \item[$(ii)$] $\operatorname{dom}\alpha=\operatorname{ran}\alpha$ and there exists a positive integer $n>1$ such that $(n,1)\in \operatorname{dom}\alpha$ and $(n,1)\alpha\in\emph{\textsf{H}}^1$;
  \item[$(iii)$] $\operatorname{dom}\alpha=\operatorname{ran}\alpha$ and there exists a positive integer $m>1$ such that $(1,m)\in \operatorname{dom}\alpha$ and $(1,m)\alpha\in\emph{\textsf{V}}^1$.
\end{itemize}
\end{theorem}

\begin{proof}
Implications $(i)\Rightarrow(ii)$ and $(i)\Rightarrow(iii)$ follow from Proposition~\ref{proposition-3.3}.

We shall prove implication $(ii)\Rightarrow(i)$ by induction in two steps. The proof of implication $(iii)\Rightarrow(i)$ is similar.

First we remark that if $(1,1)\in\operatorname{dom}\alpha$ then since $(1,1)\leqslant(i,j)$  for any $(i,j)\in\operatorname{dom}\alpha$, the definition of the semigroup $\mathscr{P\!O}\!_{\infty}(\mathbb{N}^2_{\leqslant})$ implies that $(1,1)\alpha=(1,1)$.

Now, condition $(ii)$ and Lemma~\ref{lemma-2.4} imply that $(\textsf{H}^{1}_{\operatorname{dom}\alpha})\alpha\subseteq \textsf{H}^1$. Since the set $\textsf{H}^{1}_{\operatorname{dom}\alpha}$ with the induced order from the poset $\mathbb{N}^2_{\leqslant}$ is order isomorphic to the set of all positive integers with the usual linear order, without loss of generality we may assume that $\textsf{H}^{1}_{\operatorname{dom}\alpha}=\left\{x_i^1\colon i=1,2,3,\ldots\right\}$ and $x_i^1\leqslant x_j^1$ in $\textsf{H}^{1}_{\operatorname{dom}\alpha}$ if and only if $i\leq j$. Since $(\textsf{H}^{1}_{\operatorname{dom}\alpha})\alpha\subseteq \textsf{H}^1$, Theorem~\ref{theorem-2.7}$(1)$ implies that $(x_1^1,1)\alpha\leqslant(x_1^1,1)$, and by the equality $\textsf{H}^{1}_{\operatorname{dom}\alpha}=\textsf{H}^{1}_{\operatorname{ran}\alpha}$ we get that $(x_1^1,1)\alpha=(x_1^1,1)$. Suppose that we have shown that $(x_l^1,1)\alpha=(x_l^1,1)$ for every positive integer $l<t_0$, where $t_0$ is some positive integer $\geq 2$. Then the equality $\textsf{H}^{1}_{\operatorname{dom}\alpha}=\textsf{H}^{1}_{\operatorname{ran}\alpha}$ and Theorem~\ref{theorem-2.7}$(1)$ imply that $(x_{t_0}^1,1)\alpha=(x_{t_0}^1,1)$, because $(x_{t_0}^1,1)\alpha\leqslant(x_{t_0}^1,1)$ and $(\textsf{H}^{1}_{\operatorname{dom}\alpha})\alpha\subseteq \textsf{H}^1$. Therefore, we have proved that $(x_{k}^1,1)\alpha=(x_{k}^1,1)$ for every $(x_{k},1)\in\operatorname{dom}\alpha$.

Now, we shall show that the equality $(p,q)\alpha=(p,q)$ for all positive integers $q< k_0$ and all positive integers $p$ such that $(p,q)\in\operatorname{dom}\alpha$, where $k_0$ is some positive integer $\geq 2$, implies that $(p,k_0)\alpha=(p,k_0)$ for all $(p,k_0)\in\operatorname{dom}\alpha$. Since the set $\textsf{H}^{k_0}_{\operatorname{dom}\alpha}$ with the induced order from the poset $\mathbb{N}^2_{\leqslant}$ is order isomorphic to the set of all positive integers with the usual linear order, without loss of generality we may assume that $\textsf{H}^{k_0}_{\operatorname{dom}\alpha}=\left\{x_i^{k_0}\colon i=1,2,3,\ldots\right\}$ and $x_i^{k_0}\leqslant x_j^{k_0}$ in $\textsf{H}^{k_0}_{\operatorname{dom}\alpha}$ if and only if $i\leq j$. Then the assumption of induction and Theorem~\ref{theorem-2.6}$(1)$ imply that $(\textsf{H}^{k_0}_{\operatorname{dom}\alpha})\alpha\subseteq^* \textsf{H}^{k_0}$. Theorem~\ref{theorem-2.7}$(1)$ implies that $(x_1^{k_0},{k_0})\alpha\leqslant(x_1^{k_0},{k_0})$, and by the equality $\textsf{H}^{k_0}_{\operatorname{dom}\alpha}=\textsf{H}^{k_0}_{\operatorname{ran}\alpha}$ we get that $(x_1^{k_0},{k_0})\alpha=(x_1^{k_0},{k_0})$. Suppose that we showed that $(x_l^{k_0},{k_0})\alpha=(x_l^{k_0},k_0)$ for every positive integer $l<s_0$, where $s_0$ is a some positive integer $\geq 2$. Then the equality $\textsf{H}^{k_0}_{\operatorname{dom}\alpha}=\textsf{H}^{k_0}_{\operatorname{ran}\alpha}$ and Theorem~\ref{theorem-2.7}$(1)$ imply that $(x_{s_0}^{k_0},{k_0})\alpha=(x_{s_0}^{k_0},{k_0})$, because $(x_{s_0}^{k_0},{k_0})\alpha\leqslant(x_{s_0}^{k_0},{k_0})$ and $(\textsf{H}^{k_0}_{\operatorname{dom}\alpha})\alpha\subseteq \textsf{H}^{k_0}$. Therefore, we have proved that $(x_{k}^{k_0},{k_0})\alpha=(x_{k}^{k_0},{k_0})$ for every $(x_{k}^{k_0},{k_0})\in\operatorname{dom}\alpha$.

The proof of implication $(ii)\Rightarrow(i)$ is complete.
\end{proof}

Proposition~\ref{proposition-3.3} implies the following proposition.

\begin{proposition}\label{proposition-3.5}
The subset of idempotents $E(\mathscr{P\!O}\!_{\infty}(\mathbb{N}^2_{\leqslant}))$ of the semigroup $\mathscr{P\!O}\!_{\infty}(\mathbb{N}^2_{\leqslant})$ is a commutative submonoid of $\mathscr{P\!O}\!_{\infty}(\mathbb{N}^2_{\leqslant})$ and moreover $E(\mathscr{P\!O}\!_{\infty}(\mathbb{N}^2_{\leqslant}))$ is isomorphic to the free semilattice with unit $\left(\mathscr{P}^*(\mathbb{N}^2),\cup\right)$ over the set $\mathbb{N}^2$ under the mapping $(\varepsilon)\mathfrak{h}=\mathbb{N}^2\setminus\operatorname{dom}\varepsilon$.
\end{proposition}

Later we shall need the following technical lemma.

\begin{lemma}\label{lemma-3.6}
Let $\alpha$ be an element of the semigroup $\mathscr{P\!O}\!_{\infty}(\mathbb{N}^2_{\leqslant})$. Then the following assertions hold:
\begin{itemize}
  \item[$(i)$] $\alpha=\gamma\alpha$ for some $\gamma\in\mathscr{P\!O}\!_{\infty}(\mathbb{N}^2_{\leqslant})$ if and only if the restriction $\gamma|_{\operatorname{dom}\alpha}\colon\operatorname{dom}\alpha\rightarrow \mathbb{N}^2$ is an identity partial map;
  \item[$(ii)$] $\alpha=\alpha\gamma$ for some $\gamma\in\mathscr{P\!O}\!_{\infty}(\mathbb{N}^2_{\leqslant})$ if and only if the restriction $\gamma|_{\operatorname{ran}\alpha}\colon\operatorname{ran}\alpha\rightarrow \mathbb{N}^2$ is an identity partial map
\end{itemize}
\end{lemma}

\begin{proof}
$(i)$ The implication $(\Leftarrow)$ is trivial.

$(\Rightarrow)$ Suppose that $\alpha=\gamma\alpha$ for some $\gamma\in\mathscr{P\!O}\!_{\infty}(\mathbb{N}^2_{\leqslant})$. Then we have that $\operatorname{dom}\alpha\subseteq \operatorname{dom}\gamma$ and $\operatorname{dom}\alpha\subseteq \operatorname{ran}\gamma$. Since $\gamma\colon \mathbb{N}^2\rightharpoonup\mathbb{N}^2$ is a partial bijection, the above arguments imply that $(i,j)\gamma=(i,j)$ for each $(i,j)\in\operatorname{dom}\alpha$. Indeed,  if $(i,j)\gamma=(m,n)\neq(i,j)$ for some $(i,j)\in\operatorname{dom}\alpha$ then since $\alpha\colon \mathbb{N}^2\rightharpoonup\mathbb{N}^2$ is a partial bijection we have that either
\begin{equation*}
    (i,j)\alpha=(i,j)\gamma\alpha=(m,n)\alpha\neq(i,j)\alpha, \qquad \hbox{if} \quad (m,n)\in\operatorname{dom}\alpha,
\end{equation*}
or $(m,n)\alpha$ is undefined. This completes the proof of the implication.

The proof of $(ii)$ is similar to that of $(i)$.
\end{proof}

The following theorem describes the Green relations $\mathscr{L}$, $\mathscr{R}$, $\mathscr{H}$ and $\mathscr{D}$ on the semigroup $\mathscr{P\!O}\!_{\infty}(\mathbb{N}^2_{\leqslant})$.

\begin{theorem}\label{theorem-3.7}
Let $\alpha$ and $\beta$ be elements of the semigroup $\mathscr{P\!O}\!_{\infty}(\mathbb{N}^2_{\leqslant})$. Then the following assertions hold:
\begin{itemize}
  \item[$(i)$] $\alpha\mathscr{L}\beta$ if and only if either $\alpha=\beta$ or $\alpha=\varpi\beta$;
  \item[$(ii)$] $\alpha\mathscr{R}\beta$ if and only if either $\alpha=\beta$ or $\alpha=\beta\varpi$;
  \item[$(iii)$] $\alpha\mathscr{H}\beta$ if and only if either $\alpha=\beta$ or $\alpha=\varpi\beta=\beta\varpi$;
  \item[$(iv)$] $\alpha\mathscr{D}\beta$ if and only if $\alpha=\mu\beta\nu$ for some $\mu,\nu\in H(\mathbb{I})$.
\end{itemize}
\end{theorem}

\begin{proof}
$(i)$ The implication $(\Leftarrow)$ is trivial.

$(\Rightarrow)$ Suppose that $\alpha\mathscr{L}\beta$ in the semigroup $\mathscr{P\!O}\!_{\infty}(\mathbb{N}^2_{\leqslant})$. Then there exist $\gamma,\delta\in\mathscr{P\!O}\!_{\infty}(\mathbb{N}^2_{\leqslant})$ such that $\alpha=\gamma\beta$ and $\beta=\delta\alpha$. The last equalities imply that
$\operatorname{ran}\alpha=\operatorname{ran}\beta$.

By Lemma~\ref{lemma-2.4} only one of the following cases holds:
\begin{itemize}
  \item[$(i_1)$] $(\textsf{H}^{1}_{\operatorname{dom}\alpha})\alpha\subseteq \textsf{H}^1$ and $(\textsf{H}^{1}_{\operatorname{dom}\beta})\beta\subseteq \textsf{H}^1$;
  \item[$(i_2)$] $(\textsf{H}^{1}_{\operatorname{dom}\alpha})\alpha\subseteq \textsf{H}^1$ and $(\textsf{H}^{1}_{\operatorname{dom}\beta})\beta\subseteq \textsf{V}^1$;
  \item[$(i_3)$] $(\textsf{H}^{1}_{\operatorname{dom}\alpha})\alpha\subseteq \textsf{V}^1$ and $(\textsf{H}^{1}_{\operatorname{dom}\beta})\beta\subseteq \textsf{H}^1$;
  \item[$(i_4)$] $(\textsf{H}^{1}_{\operatorname{dom}\alpha})\alpha\subseteq \textsf{V}^1$ and $(\textsf{H}^{1}_{\operatorname{dom}\beta})\beta\subseteq \textsf{V}^1$.
\end{itemize}

Suppose that case $(i_1)$ holds. Then the equalities $\alpha=\gamma\beta$ and $\beta=\delta\alpha$ imply that
\begin{equation}\label{eq-2.1}
    (\textsf{H}^{1}_{\operatorname{dom}\gamma})\gamma\subseteq \textsf{H}^1 \qquad \hbox{ and } \qquad
    (\textsf{H}^{1}_{\operatorname{dom}\delta})\delta\subseteq \textsf{H}^1,
\end{equation}
and moreover we have that $\alpha=\gamma\delta\alpha$ and $\beta=\delta\gamma\beta$. Hence by Lemma~\ref{lemma-3.6} we have that the restrictions $(\gamma\delta)|_{\operatorname{dom}\alpha}\colon\operatorname{dom}\alpha\rightarrow \mathbb{N}^2$ and $(\delta\gamma)|_{\operatorname{dom}\beta}\colon\operatorname{dom}\beta\rightarrow \mathbb{N}^2$ are identity partial maps. Then by condition $(\ref{eq-2.1})$ we obtain that the restrictions $\gamma|_{\operatorname{dom}\alpha}\colon\operatorname{dom}\alpha\rightarrow \mathbb{N}^2$ and $\delta|_{\operatorname{dom}\beta}\colon\operatorname{dom}\beta\rightarrow \mathbb{N}^2$ are also identity partial maps. Indeed, other wise there exists $(i,j)\in\operatorname{dom}\alpha$ such that either $(i,j)\gamma\nleqslant(i,j)$ or $(i,j)\delta\nleqslant(i,j)$, which contradicts Theorem~\ref{theorem-2.7}$(1)$. Thus, the above arguments imply that in case $(i_1)$ we have that $\alpha=\beta$.

Suppose that case $(i_2)$ holds. Then we have that $\alpha=\gamma\beta=\gamma\mathbb{I}\beta=\gamma(\varpi\varpi)\beta=(\gamma\varpi)(\varpi\beta)$ and $\varpi\beta=(\varpi\delta)\alpha$. Hence we get that $\alpha\mathscr{L}(\varpi\beta)$, $(\textsf{H}^{1}_{\operatorname{dom}\alpha})\alpha\subseteq \textsf{H}^1$ and $(\textsf{H}^{1}_{\operatorname{dom}(\varpi\beta)})\varpi\beta\subseteq \textsf{H}^1$. Then we apply case $(i_1)$ for elements $\alpha$ and $\varpi\beta$ and obtain that $\alpha=\varpi\beta$.

In case $(i_3)$ the proof of the equality $\alpha=\varpi\beta$ is similar to case $(i_2)$.

Suppose that case $(i_4)$ holds. Then the equalities $\alpha=\gamma\beta$ and $\beta=\delta\alpha$ imply that $\alpha\varpi=\gamma(\beta\varpi)$ and $\beta\varpi=\delta(\alpha\varpi)$, which implies that $(\alpha\varpi)\mathscr{L}(\beta\varpi)$. Since for the elements $\alpha\varpi$ and $\beta\varpi$ of the semigroup $\mathscr{P\!O}\!_{\infty}(\mathbb{N}^2_{\leqslant})$ case $(i_1)$ holds, $\alpha\varpi=\beta\varpi$ and hence $\alpha=\alpha\varpi\varpi=\alpha\varpi\varpi=\beta$, which completes the proof of $(i)$.

The proof of assertion $(ii)$ is dual to that of $(i)$.

Assertion $(iii)$ follows from $(i)$  $(ii)$.

$(iv)$ Suppose that $\alpha\mathscr{D}\beta$ in $\mathscr{P\!O}\!_{\infty}(\mathbb{N}^2_{\leqslant})$. Then there exists $\gamma\in\mathscr{P\!O}\!_{\infty}(\mathbb{N}^2_{\leqslant})$ such that $\alpha\mathscr{L}\gamma$ and $\gamma\mathscr{R}\beta$. By Proposition~\ref{proposition-3.1} the group of units $H(\mathbb{I})$ of the semigroup $\mathscr{P\!O}\!_{\infty}(\mathbb{N}^2_{\leqslant})$ has two distinct elements $\mathbb{I}$ and $\varpi$. By $(i)$,  $(ii)$, there exist $\mu,\nu\in H(\mathbb{I})$ such that $\alpha=\mu\gamma$ and $\gamma=\beta\nu$ and hence $\alpha=\mu\beta\nu$. Converse, suppose that $\alpha=\mu\beta\nu$ for some $\mu,\nu\in H(\mathbb{I})$. Then by $(i)$,  $(ii)$, we have that $\alpha\mathscr{L}(\beta\nu)$ and $\beta\mathscr{R}(\beta\nu)$, and hence $\alpha\mathscr{D}\beta$.
\end{proof}

Theorem~\ref{theorem-3.7} implies Corollary~\ref{corollary-3.8} which gives the inner characterization of  the Green relations $\mathscr{L}$, $\mathscr{R}$, $\mathscr{H}$ and $\mathscr{D}$ on the semigroup $\mathscr{P\!O}\!_{\infty}(\mathbb{N}^2_{\leqslant})$ as partial permutations of the poset $\mathbb{N}^2_{\leqslant}$.

\begin{corollary}\label{corollary-3.8}
\begin{itemize}
  \item[$(i)$] Every $\mathscr{L}$-class of $\mathscr{P\!O}\!_{\infty}(\mathbb{N}^2_{\leqslant})$ contains two distinct elements.
  \item[$(ii)$] Every $\mathscr{R}$-class of $\mathscr{P\!O}\!_{\infty}(\mathbb{N}^2_{\leqslant})$ contains  two distinct elements.
  \item[$(iii)$] Every $\mathscr{H}$-class of $\mathscr{P\!O}\!_{\infty}(\mathbb{N}^2_{\leqslant})$ contains at most two distinct elements.
  \item[$(iv)$] The $\mathscr{H}$-class of $\mathscr{P\!O}\!_{\infty}(\mathbb{N}^2_{\leqslant})$ which contains an element $\alpha$ consists of two distinct elements if and only if $\operatorname{dom}\alpha=\overline{\operatorname{dom}\alpha}$, $\operatorname{ran}\alpha=\overline{\operatorname{ran}\alpha}$ and $(\overline{(i,j)})\alpha=\overline{(i,j)\alpha}$ for each $(i,j)\in\operatorname{dom}\alpha$, and the $\mathscr{H}$-class of $\alpha$ is a singleton in the other case.
  \item[$(v)$] The $\mathscr{H}$-class of $\mathscr{P\!O}\!_{\infty}(\mathbb{N}^2_{\leqslant})$ which contains an idempotent $\varepsilon$ consists of two distinct elements if and only if $\operatorname{dom}\varepsilon=\overline{\operatorname{dom}\varepsilon}$.
  \item[$(vi)$] The $\mathscr{H}$-class of $\mathscr{P\!O}\!_{\infty}(\mathbb{N}^2_{\leqslant})$ which contains an idempotent $\varepsilon$ is a singleton if and only if $\operatorname{dom}\varepsilon\neq\overline{\operatorname{dom}\varepsilon}$.
  \item[$(vii)$] Every $\mathscr{D}$-class of $\mathscr{P\!O}\!_{\infty}(\mathbb{N}^2_{\leqslant})$ contains either two or four distinct elements.
  \item[$(viii)$] A $\mathscr{D}$-class of $\mathscr{P\!O}\!_{\infty}(\mathbb{N}^2_{\leqslant})$ has two distinct elements if and only if it contains only one $\mathscr{H}$-class.
  \item[$(ix)$] A $\mathscr{D}$-class of $\mathscr{P\!O}\!_{\infty}(\mathbb{N}^2_{\leqslant})$ has two distinct elements if and only if it contains a non-singleton $\mathscr{H}$-class.
  \item[$(x)$] A $\mathscr{D}$-class of $\mathscr{P\!O}\!_{\infty}(\mathbb{N}^2_{\leqslant})$ has four distinct elements if and only every its $\mathscr{H}$-class is singleton.
   \item[$(xi)$] A $\mathscr{D}$-class of $\mathscr{P\!O}\!_{\infty}(\mathbb{N}^2_{\leqslant})$ has four distinct elements if and only it contains a singleton $\mathscr{H}$-class.
  \item[$(xii)$] The $\mathscr{D}$-class of $\mathscr{P\!O}\!_{\infty}(\mathbb{N}^2_{\leqslant})$ which contains an idempotent $\varepsilon$ consists of two distinct elements if and only if $\operatorname{dom}\varepsilon=\overline{\operatorname{dom}\varepsilon}$.
  \item[$(xiii)$] The $\mathscr{D}$-class of $\mathscr{P\!O}\!_{\infty}(\mathbb{N}^2_{\leqslant})$ which contains an idempotent $\varepsilon$ consists of four distinct elements if and only if $\operatorname{dom}\varepsilon\neq\overline{\operatorname{dom}\varepsilon}$.
\end{itemize}
\end{corollary}

\begin{proof}
Statements $(i)$, $(ii)$ and $(iii)$  are trivial and they follow from the equality $\varpi\varpi=\mathbb{I}$ and the corresponding statements of Theorem~\ref{theorem-3.7}.

$(iv)$ By $(i)$ and $(ii)$ we have that the $\mathscr{H}$-class of $\mathscr{P\!O}\!_{\infty}(\mathbb{N}^2_{\leqslant})$ which contains an element $\alpha$ contains at most two distinct elements.

$(\Rightarrow)$ Assume that $\alpha\mathscr{H}\beta$ in $\mathscr{P\!O}\!_{\infty}(\mathbb{N}^2_{\leqslant})$ and $\alpha\neq\beta$. By Theorem~\ref{theorem-3.7}$(iii)$, $\beta=\alpha\varpi=\varpi\alpha$. Then by the definition of $\varpi$ we get that $\operatorname{dom}\beta=\operatorname{dom}\alpha=\overline{\operatorname{dom}\alpha}$ and $\operatorname{ran}\beta=\operatorname{ran}\alpha=\overline{\operatorname{ran}\alpha}$. If $(i,j)\in\operatorname{dom}\alpha$ and $(i,j)\alpha=(m,n)$ then
\begin{equation*}
    (n,m)=(m,n)\varpi=(i,j)\alpha\varpi=(i,j)\beta=(i,j)\varpi\alpha=(j,i)\alpha.
\end{equation*}
This completes the proof of the implication.

The converse implication is trivial, and the last statement of item $(iv)$ follows from the above part of its proof.

$(v)$ If $\operatorname{dom}\varepsilon=\overline{\operatorname{dom}\varepsilon}$ then $\varepsilon\varpi=\varpi\varepsilon\neq\varepsilon$. Conversely, suppose that $\varepsilon\varpi=\varpi\varepsilon\neq\varepsilon$.  Since $\operatorname{dom}\varpi=\operatorname{ran}\varpi=\mathbb{N}\times\mathbb{N}$ and $\operatorname{dom}\varepsilon=\operatorname{ran}\varepsilon$, the equality $\varepsilon\varpi=\varpi\varepsilon$ implies that $\operatorname{dom}(\varepsilon\varpi)=\operatorname{dom}\varepsilon=\operatorname{ran}\varepsilon=\operatorname{ran}(\varpi\varepsilon)$, and hence the definition of the element $\varpi\in H(\mathbb{I})$ implies that $\operatorname{dom}\varepsilon=\overline{\operatorname{dom}\varepsilon}$.

Statement $(vi)$ follows from items $(iii)$, $(v)$.

$(vii)$ Theorem~\ref{theorem-3.7}$(iv)$ and $(i)$, $(ii)$ imply that every $\mathscr{D}$-class of the semigroup $\mathscr{P\!O}\!_{\infty}(\mathbb{N}^2_{\leqslant})$ contains at most four and at least two distinct elements. Suppose to the contrary that there exists a $\mathscr{D}$-class $D_\alpha$ in $\mathscr{P\!O}\!_{\infty}(\mathbb{N}^2_{\leqslant})$ which contains three distinct elements such that $\alpha\in D_\alpha$ for some element $\alpha$ of the semigroup $\mathscr{P\!O}\!_{\infty}(\mathbb{N}^2_{\leqslant})$. By Theorem~\ref{theorem-3.7}$(iv)$, $\varpi\alpha,\alpha\varpi,\varpi\alpha\varpi\in D_\alpha$. Since $\varpi\gamma\neq\gamma\neq\gamma\varpi$ for any $\gamma\in\mathscr{P\!O}\!_{\infty}(\mathbb{N}^2_{\leqslant})$, we have that $\varpi\alpha=\alpha\varpi$ or $\alpha=\varpi\alpha\varpi$. If $\varpi\alpha=\alpha\varpi$ then the definition of the element $\varpi$ of $\mathscr{P\!O}\!_{\infty}(\mathbb{N}^2_{\leqslant})$ implies that $\alpha=\varpi\varpi\alpha=\varpi\alpha\varpi$. Similarly, if $\alpha=\varpi\alpha\varpi$ then $\varpi\alpha=\varpi\varpi\alpha\varpi=\alpha\varpi$. This completes the proof of the statement.

$(viii)$ $(\Rightarrow)$ Assume that a $\mathscr{D}$-class of $\mathscr{P\!O}\!_{\infty}(\mathbb{N}^2_{\leqslant})$ has two distinct elements and it contains $\alpha$. Then the proof of item $(vii)$ implies that $\varpi\alpha=\alpha\varpi$ and $\alpha=\varpi\alpha\varpi$.  By Theorem~\ref{theorem-3.7}$(iv)$ we have that $D_\alpha=H_\alpha$.

Implication $(\Leftarrow)$ is trivial.

$(ix)$ Implication $(\Rightarrow)$ follows form item $(viii)$.

$(\Leftarrow)$ Assume that there exists a $\mathscr{D}$-class of $\mathscr{P\!O}\!_{\infty}(\mathbb{N}^2_{\leqslant})$ which contains a non-singleton $\mathscr{H}$-class $H_\alpha$ of $\mathscr{P\!O}\!_{\infty}(\mathbb{N}^2_{\leqslant})$ for some $\alpha\in\mathscr{P\!O}\!_{\infty}(\mathbb{N}^2_{\leqslant})$. By Theorem~\ref{theorem-3.7}$(iii)$ we have that $H_\alpha=\left\{\alpha,\alpha\varpi\right\}$ and $\alpha\neq\alpha\varpi=\varpi\alpha$. Then the last equality implies that $\alpha=\varpi\alpha\varpi$. Hence by Theorem~\ref{theorem-3.7}$(iv)$, $D_\alpha=H_\alpha$, which complete the proof of the implication.

Statement $(x)$ follows from  $(viii)$,  $(ix)$.

$(xi)$ By Theorem~2.3 of \cite{Clifford-Preston-1961-1967} any two $\mathscr{H}$-classes of an arbitrary $\mathscr{D}$-class are of the same cardinality. Now, we apply statement $(x)$.

Statement $(xii)$ follows from $(viii)$,  $(v)$.

Items $(x)$ and  $(vi)$ imply statement $(xiii)$.
\end{proof}

We need the following three lemmas.

\begin{lemma}\label{lemma-3.9}
Let $\alpha,\beta$ and $\gamma$ be elements of the semigroup $\mathscr{P\!O}\!_{\infty}(\mathbb{N}^2_{\leqslant})$ such that $\alpha=\beta\alpha\gamma$. Then the following statements hold:
\begin{itemize}
  \item[$(i)$] if \emph{$(\textsf{H}^{1}_{\operatorname{dom}\beta})\beta\subseteq \textsf{H}^1$} then the restrictions $\beta|_{\operatorname{dom}\alpha}\colon \operatorname{dom}\alpha\rightharpoonup \mathbb{N}\times\mathbb{N}$ and $\gamma|_{\operatorname{ran}\alpha}\colon \operatorname{ran}\alpha\rightharpoonup \mathbb{N}\times\mathbb{N}$ are identity partial maps;
  \item[$(ii)$] if \emph{$(\textsf{H}^{1}_{\operatorname{dom}\beta})\beta\subseteq \textsf{V}^1$} then $(i,j)\beta=(j,i)$ for each $(i,j)\in\operatorname{dom}\alpha$ and $(m,n)\gamma=(n,m)$ for each $(m,n)\in\operatorname{ran}\alpha$; and moreover in this case we have that $\operatorname{dom}\alpha=\overline{\operatorname{dom}\alpha}$, $\operatorname{ran}\alpha=\overline{\operatorname{ran}\alpha}$ and $(j,i)\alpha=\overline{(i,j)\alpha}$ for any $(i,j)\in\operatorname{dom}\alpha$, i.e., $\alpha=\varpi\alpha\varpi$.
\end{itemize}
\end{lemma}

\begin{proof}
$(i)$ Assume that the inclusion $(\textsf{H}^{1}_{\operatorname{dom}\beta})\beta\subseteq \textsf{H}^1$ holds. Then one of the following cases holds:
\begin{itemize}
  \item[$(1)$] $(\textsf{H}^{1}_{\operatorname{dom}\alpha})\alpha\subseteq \textsf{H}^1$;
  \item[$(2)$] $(\textsf{H}^{1}_{\operatorname{dom}\alpha})\alpha\subseteq \textsf{V}^1$.
\end{itemize}

If case $(1)$ holds then the equality $\alpha=\beta\alpha\gamma$ and Lemma~\ref{lemma-2.4} imply that $(\textsf{H}^{1}_{\operatorname{dom}\gamma})\gamma\subseteq \textsf{H}^1$. By Theorem~\ref{theorem-2.7}$(1)$, $(i,j)\beta\leqslant(i,j)$ for any $(i,j)\in\operatorname{dom}\beta$ and $(m,n)\gamma\leqslant(m,n)$ for any $(m,n)\in\operatorname{dom}\gamma$. Suppose that $(i,j)\beta<(i,j)$ for some $(i,j)\in\operatorname{dom}\alpha$. Then we have that
\begin{equation*}
    (i,j)\alpha=(i,j)\beta\alpha\gamma<(i,j)\alpha\gamma\leqslant(i,j)\alpha,
\end{equation*}
which contradicts the equality $\alpha=\beta\alpha\gamma$. The obtained contradiction implies that the restriction $\beta|_{\operatorname{dom}\alpha}\colon \operatorname{dom}\alpha\rightharpoonup \mathbb{N}\times\mathbb{N}$ is an identity partial map. This and the equality $\alpha=\beta\alpha\gamma$ imply that the restriction $\gamma|_{\operatorname{ran}\alpha}\colon \operatorname{ran}\alpha\rightharpoonup \mathbb{N}\times\mathbb{N}$ is an identity partial map too.

Suppose that case $(2)$ holds. Then we have that $(\textsf{H}^{1}_{\operatorname{dom}\alpha})\alpha\varpi\subseteq \textsf{H}^1$. Now, the equality $\alpha=\beta\alpha\gamma$ and the definition of the element $\varpi$ the semigroup $\mathscr{P\!O}\!_{\infty}(\mathbb{N}^2_{\leqslant})$ imply that
\begin{equation*}
    \alpha\varpi=\beta\alpha\gamma\varpi=\beta(\alpha\varpi)(\varpi\gamma\varpi).
\end{equation*}
Then we apply case $(1)$. This completes the proof of $(i)$.

$(ii)$ Assume that the inclusion $(\textsf{H}^{1}_{\operatorname{dom}\beta})\beta\subseteq \textsf{V}^1$ holds. Then the equality $\alpha=\beta\alpha\gamma$ implies  that $\alpha=\beta\beta\alpha\gamma\gamma$ and the inclusion $(\textsf{H}^{1}_{\operatorname{dom}\beta})\beta\subseteq \textsf{V}^1$ implies that $(\textsf{H}^{1}_{\operatorname{dom}(\beta\beta)})\beta\beta\subseteq \textsf{H}^1$. Now, by $(i)$, the restrictions $(\beta\beta)|_{\operatorname{dom}\alpha}\colon \operatorname{dom}\alpha\rightharpoonup \mathbb{N}\times\mathbb{N}$ and $(\gamma\gamma)|_{\operatorname{ran}\alpha}\colon \operatorname{ran}\alpha\rightharpoonup \mathbb{N}\times\mathbb{N}$ are identity partial maps. Since $(\textsf{H}^{1}_{\operatorname{dom}\beta})\beta\subseteq \textsf{V}^1$, Theorem~\ref{theorem-2.7}$(2)$ implies that $(i,j)\beta\leqslant(j,i)$ for any $(i,j)\in\operatorname{dom}\alpha$. Suppose that $(i,j)\beta<(j,i)$ for some $(i,j)\in\operatorname{dom}\alpha$. Again, by Theorem~\ref{theorem-2.7}$(2)$ we get that  $(j,i)\beta\leqslant(i,j)$ and hence we have that $(i,j)=(i,j)\beta\beta<(j,i)\beta\leqslant(i,j)$, a contradiction. The obtained contradiction implies that $(i,j)\beta=(j,i)$ for each $(i,j)\in\operatorname{dom}\alpha$. Next, the inclusion $(\textsf{H}^{1}_{\operatorname{dom}\beta})\beta\subseteq \textsf{V}^1$ and the equality $\alpha=\beta\alpha\gamma$ imply that $(\textsf{H}^{1}_{\operatorname{dom}\gamma})\gamma\subseteq \textsf{V}^1$. Then the similar arguments as in the above part of the proof imply that $(m,n)\gamma=(n,m)$ for each $(m,n)\in\operatorname{ran}\alpha$.

Now, the property that $(i,j)\beta=(j,i)$ for each $(i,j)\in\operatorname{dom}\alpha$ and $(m,n)\gamma=(n,m)$ for each $(m,n)\in\operatorname{ran}\alpha$, and the equality $\alpha=\beta\alpha\gamma$ imply that $\operatorname{dom}\alpha=\overline{\operatorname{dom}\alpha}$ and $\operatorname{ran}\alpha=\overline{\operatorname{ran}\alpha}$. Fix an arbitrary $(i,j)\in\operatorname{dom}\alpha$. Put $(m,n)=(i,j)\alpha$. Then the above part of the proof of this item implies that $(m,n)=(i,j)\alpha=(i,j)\beta\alpha\gamma=(j,i)\alpha\gamma$ and hence $(n,m)=(m,n)\varpi=(j,i)\alpha\gamma\varpi=(j,i)\alpha$.
\end{proof}

\begin{lemma}\label{lemma-3.10}
Let $\alpha$ and $\beta$ be elements of the semigroup $\mathscr{P\!O}\!_{\infty}(\mathbb{N}^2_{\leqslant})$ and $A$ be a cofinite subset of $\mathbb{N}\times\mathbb{N}$. If the restriction $(\alpha\beta)|_A\colon A\rightharpoonup \mathbb{N}\times\mathbb{N}$ is an identity partial map then one of the following conditions holds:
\begin{itemize}
  \item[$(i)$] the restrictions $\alpha|_A\colon A\rightharpoonup \mathbb{N}\times\mathbb{N}$ and $\beta|_A\colon A\rightharpoonup \mathbb{N}\times\mathbb{N}$ are identity partial maps;
  \item[$(ii)$] $(i,j)\alpha=(j,i)$ for all $(i,j)\in A$ and $(m,n)\beta=(n,m)$ for all $(m,n)\in \overline{A}$.
\end{itemize}
\end{lemma}

\begin{proof}
By Lemma~\ref{lemma-2.4} we have that either $(\textsf{H}^{1}_{\operatorname{dom}\alpha})\alpha\subseteq \textsf{H}^1$ or $(\textsf{H}^1_{\operatorname{dom}\alpha})\alpha\subseteq \textsf{V}^1$. Suppose that the inclusion $(\textsf{H}^{1}_{\operatorname{dom}\alpha})\alpha\subseteq \textsf{H}^1$ holds. Then the definition of the semigroup $\mathscr{P\!O}\!_{\infty}(\mathbb{N}^2_{\leqslant})$ implies that $(\textsf{H}^{1}_{\operatorname{dom}\beta})\beta\subseteq \textsf{H}^1$. By Theorem~\ref{theorem-2.7}$(1)$ we have that  $(i,j)\alpha\leqslant(i,j)$ for any $(i,j)\in\operatorname{dom}\alpha$ and  $(m,n)\beta\leqslant(m,n)$ for any $(m,n)\in\operatorname{dom}\beta$. Suppose that $(i,j)\alpha<(i,j)$ for some $(i,j)\in A$. Then we have that
\begin{equation*}
    (i,j)=(i,j)\alpha\beta<(i,j)\beta\leqslant(i,j),
\end{equation*}
which contradicts the assumption that the restriction $(\alpha\beta)|_A\colon A\rightharpoonup \mathbb{N}\times\mathbb{N}$ is an identity partial map. Hence the restriction $\alpha|_A\colon A\rightharpoonup \mathbb{N}\times\mathbb{N}$ is an identity partial map. Similar arguments imply that the restriction $\beta|_A\colon A\rightharpoonup \mathbb{N}\times\mathbb{N}$ is also an identity partial map. Thus, in the case when $(\textsf{H}^{1}_{\operatorname{dom}\alpha})\alpha\subseteq \textsf{H}^1$, item $(i)$ holds.

Suppose that the inclusion $(\textsf{H}^{1}_{\operatorname{dom}\alpha})\alpha\subseteq \textsf{V}^1$ holds. By the definition of the semigroup $\mathscr{P\!O}\!_{\infty}(\mathbb{N}^2_{\leqslant})$ we have that $(\textsf{H}^{1}_{\operatorname{dom}\beta})\beta\subseteq \textsf{V}^1$, $\alpha\beta=(\alpha\varpi)(\varpi\beta)$, $(\textsf{H}^{1}_{\operatorname{dom}(\alpha\varpi)})\alpha\varpi\subseteq \textsf{H}^1$ and $(\textsf{H}^{1}_{\operatorname{dom}(\varpi\beta)})\varpi\beta\subseteq \textsf{H}^1$. Then the previous part of the proof implies that the restrictions $(\alpha\varpi)|_A\colon A\rightharpoonup \mathbb{N}\times\mathbb{N}$ and $(\varpi\beta)|_A\colon A\rightharpoonup \mathbb{N}\times\mathbb{N}$ are identity partial maps. Since $(\alpha\varpi)\varpi=\alpha$ and $\varpi(\varpi\beta)=\beta$, the inclusion $(\textsf{H}^{1}_{\operatorname{dom}\alpha})\alpha\subseteq \textsf{V}^1$ implies that $(ii)$ holds.
\end{proof}

\begin{lemma}\label{lemma-3.11}
Let $\alpha$ and $\beta$ be elements of the semigroup $\mathscr{P\!O}\!_{\infty}(\mathbb{N}^2_{\leqslant})$ and $A$ be a cofinite subset of $\mathbb{N}\times\mathbb{N}$. If $(i,j)\alpha\beta=(j,i)$ for all $(i,j)\in A$, then one of the following conditions holds:
\begin{itemize}
  \item[$(i)$] the restriction $\alpha|_A\colon A\rightharpoonup \mathbb{N}\times\mathbb{N}$ is an identity partial map and $(m,n)\beta=(n,m)$ for all $(m,n)\in A$;
  \item[$(ii)$] $(i,j)\alpha=(j,i)$ for all $(i,j)\in A$ and $\beta|_{\overline{A}}\colon \overline{A}\rightharpoonup \mathbb{N}\times\mathbb{N}$ is an identity partial map.
\end{itemize}
\end{lemma}

\begin{proof}
The assumption of the lemma implies that the restriction $\alpha(\beta\varpi)|_A\colon \rightharpoonup \mathbb{N}\times\mathbb{N}$ is an identity partial map. Hence by Lemma~\ref{lemma-3.10} only one of the following conditions holds:
\begin{itemize}
  \item[$(1)$] the restrictions $\alpha|_A\colon A\rightharpoonup \mathbb{N}\times\mathbb{N}$ and $(\beta\varpi)|_A\colon A\rightharpoonup \mathbb{N}\times\mathbb{N}$ are identity partial maps;
  \item[$(2)$] $(i,j)\alpha=(j,i)$ for all $(i,j)\in A$ and $(m,n)\beta\varpi=(n,m)$ for all $(m,n)\in \overline{A}$.
\end{itemize}
Since $(\beta\varpi)\varpi=\beta$, the above arguments imply the statement of the lemma.
\end{proof}

Elementary calculations and the definition of the semigroup $\mathscr{P\!O}\!_{\infty}(\mathbb{N}^2_{\leqslant})$ imply the following proposition.

\begin{proposition}\label{proposition-3.12}
Let $\alpha$ and $\beta$ be elements of the semigroup $\mathscr{P\!O}\!_{\infty}(\mathbb{N}^2_{\leqslant})$. Then the following assertions hold:
\begin{itemize}
  \item[$(i)$] if the restriction $\beta|_{\operatorname{ran}\alpha}\colon \operatorname{ran}\alpha\rightharpoonup \mathbb{N}\times\mathbb{N}$ is an identity partial map then $\alpha\beta=\alpha\mathbb{I}=\alpha$;
  \item[$(ii)$] if the restriction $\beta|_{\operatorname{dom}\alpha}\colon \operatorname{dom}\alpha\rightharpoonup \mathbb{N}\times\mathbb{N}$ is an identity partial map then $\beta\alpha=\mathbb{I}\alpha=\alpha$;
  \item[$(iii)$] if $(m,n)\beta=(n,m)$ for all $(m,n)\in \operatorname{ran}\alpha$ then $\alpha\beta=\alpha\varpi$;
  \item[$(iv)$] if $(m,n)\beta=(n,m)$ for all $(m,n)\in \overline{\operatorname{dom}\alpha}$ then $\beta\alpha=\varpi\alpha$.
\end{itemize}
\end{proposition}

\begin{theorem}\label{theorem-3.12}
$\mathscr{D}=\mathscr{J}$ in $\mathscr{P\!O}\!_{\infty}(\mathbb{N}^2_{\leqslant})$.
\end{theorem}

\begin{proof}
The inclusion $\mathscr{D}\subseteq\mathscr{J}$ is trivial.

Fix any $\alpha,\beta\in\mathscr{P\!O}\!_{\infty}(\mathbb{N}^2_{\leqslant})$ such that $\alpha\mathscr{J}\beta$. Then there exist $\gamma_\alpha,\delta_\alpha, \gamma_\beta,\delta_\beta\in\mathscr{P\!O}\!_{\infty}(\mathbb{N}^2_{\leqslant})$ such that $\alpha=\gamma_\alpha\beta\delta_\alpha$ and $\beta=\gamma_\beta\alpha\delta_\beta$ (see \cite{GreenJ-1951} or \cite[Section~II.1]{Grillet-1995}). Hence we have that $\alpha=\gamma_\alpha\gamma_\beta\alpha\delta_\beta\delta_\alpha$ and $\beta=\gamma_\beta\gamma_\alpha\beta\delta_\alpha\delta_\beta$.

Suppose that $(\textsf{H}^{1}_{\operatorname{dom}(\gamma_\alpha\gamma_\beta)})\gamma_\alpha\gamma_\beta\subseteq \textsf{H}^1$. By Proposition~\ref{proposition-2.10}, $(\textsf{H}^{1}_{\operatorname{dom}(\gamma_\beta\gamma_\alpha)})\gamma_\beta\gamma_\alpha\subseteq \textsf{H}^1$. Lemma~\ref{lemma-3.9}$(i)$ implies that the restrictions $(\gamma_\alpha\gamma_\beta)|_{\operatorname{dom}\alpha}\colon \operatorname{dom}\alpha\rightharpoonup \mathbb{N}\times\mathbb{N}$, $(\delta_\beta\delta_\alpha)|_{\operatorname{ran}\alpha}\colon \operatorname{ran}\alpha\rightharpoonup \mathbb{N}\times\mathbb{N}$, $(\gamma_\beta\gamma_\alpha)|_{\operatorname{dom}\beta}\colon \operatorname{dom}\beta\rightharpoonup \mathbb{N}\times\mathbb{N}$ and $(\delta_\alpha\delta_\beta)|_{\operatorname{ran}\beta}\colon \operatorname{ran}\beta\rightharpoonup \mathbb{N}\times\mathbb{N}$ are identity partial maps. Then by Lemma~\ref{lemma-3.10} and Proposition~\ref{proposition-3.12} there exist $\omega_1,\omega_2\in H(\mathbb{I})$ such that $\gamma_\beta\alpha=\omega_1\alpha$, $\alpha\delta_\beta=\alpha\omega_2$, $\gamma_\alpha\beta=\omega_1\beta$ and $\beta\delta_\alpha=\beta\omega_2$. This implies that
\begin{equation*}
\alpha=\gamma_\alpha\beta\delta_\alpha=\omega_1\beta\delta_\alpha=\omega_1\beta\omega_2 \qquad \hbox{and}\qquad \beta=\gamma_\beta\alpha\delta_\beta=\omega_1\alpha\delta_\beta=\omega_1\alpha\omega_2,
\end{equation*}
and hence by Theorem~\ref{theorem-3.7} we get that $\alpha\mathscr{D}\beta$.

Suppose that $(\textsf{H}^{1}_{\operatorname{dom}(\gamma_\alpha\gamma_\beta)})\gamma_\alpha\gamma_\beta\subseteq \textsf{V}^1$. Then by Proposition~\ref{proposition-2.10} and Lemma~\ref{lemma-2.4} we have that  $(\textsf{H}^{1}_{\operatorname{dom}(\gamma_\beta\gamma_\alpha)})\gamma_\beta\gamma_\alpha\subseteq \textsf{V}^1$. Now, as in the above part of the proof the statement of the theorem follows from Lemma~\ref{lemma-3.9}$(ii)$, Lemma~\ref{lemma-3.11} and Proposition~\ref{proposition-3.12}.
\end{proof}


\section*{Acknowledgements}

The author acknowledges T. Banakh and A. Ravsky for their comments and suggestions.

\end{document}